\documentclass[12pt,reqno]{amsart}

\usepackage{tikz}
\usepackage{xcolor}

\usepackage{latexsym,amsmath,amssymb,epic}
\usepackage{latexsym,amssymb,epic}
\usepackage{amsmath,amsthm}
\usepackage{color}

\usepackage[left=3.75cm,right=3.75cm]{geometry}
\usepackage{etoolbox}
\usepackage{enumitem}
\usepackage[noadjust]{cite}
\usepackage[pagebackref]{hyperref}

\hypersetup{
	colorlinks=true, 
	linktoc=all, 
	linkcolor=blue} 

\allowdisplaybreaks[4]

\numberwithin{equation}{section}

\theoremstyle{theorem}
\newtheorem{theorem}{Theorem}
\newtheorem*{theorem*}{Theorem}
\newtheorem{corollary}[theorem]{Corollary}
\newtheorem{lemma}[theorem]{Lemma}

\theoremstyle{definition}

\newtheorem{example}{Example}
\newtheorem*{example*}{Example}

\newtheorem*{conjecture*}{Conjecture}
\newtheorem{remark}{Remark}
\newtheorem*{remark*}{Remark}
\newtheorem*{remarks*}{Remarks}

\makeatletter
\patchcmd{\section}{\scshape}{\bfseries\boldmath}{}{}
\patchcmd{\subsection}{\bfseries}{\bfseries\boldmath}{}{}
\renewcommand{\@secnumfont}{\bfseries}
\makeatother

\newcommand{\ord}{\operatorname{ord}}

\begin{document}
	
\title[Partitions into odd parts with designated summands]{An infinite family of internal congruences modulo powers of $2$ for partitions into odd parts with designated summands}
	
\author[S. Chern]{Shane Chern}
\address[S. Chern]{Department of Mathematics and Statistics, Dalhousie University, Halifax, NS, B3H 4R2, Canada}
\email{chenxiaohang92@gmail.com}

\author[J. A. Sellers]{James A. Sellers}
\address[J. A. Sellers]{Department of Mathematics and Statistics, University of Minnesota Duluth, Duluth, MN 55812, USA}
\email{jsellers@d.umn.edu}
	
	
	
\maketitle

\begin{abstract}
In 2002, Andrews, Lewis, and Lovejoy introduced the combinatorial objects which they called \emph{partitions with designated summands}.  These are built by taking unrestricted integer partitions and designating exactly one of each occurrence of a part.  In that same work, Andrews, Lewis, and Lovejoy also studied such partitions wherein all parts must be odd, and they denoted the number of such partitions of size $n$ by the function $PDO(n)$.  Since then, numerous authors have proven a variety of divisibility properties satisfied by $PDO(n)$. Recently, the second author proved the following internal congruences satisfied by $PDO(n)$:  \textit{For all $n\geq 0$},
\begin{align*}
PDO(4n) &\equiv PDO(n) \pmod{4},\\
PDO(16n) &\equiv PDO(4n) \pmod{8}.  
\end{align*}
In this work, we significantly extend these internal congruence results by proving the following new infinite family of congruences:  
\textit{For all $k\geq 0$ and all $n\geq 0$}, 
$$PDO(2^{2k+3}n) \equiv PDO(2^{2k+1}n) \pmod{2^{2k+3}}.$$ 
We utilize several classical tools to prove this family, including generating function dissections via the unitizing operator of degree two, various modular relations and recurrences involving a Hauptmodul on the classical modular curve $X_0(6)$, and an induction argument which provides the final step in proving the necessary divisibilities. It is notable that the construction of each $2$-dissection slice of our generating function bears an entirely different nature to those studied in the past literature.
\end{abstract}

\section{Introduction}  
In 2002, Andrews, Lewis, and Lovejoy \cite{ALL} introduced the combinatorial objects which they called {\it partitions with designated summands}.  These are built by taking unrestricted integer partitions and designating exactly one of each occurrence of a part.  For example, there are ten partitions with designated summands of size $4$:  
\begin{gather*}
    4',\ \ \  3'+1',\ \ \  2'+2,\ \ \  2+2',\ \ \  2'+1'+1,\ \ \  2'+1+1',\\
    1'+1+1+1,\ \ \  1+1'+1+1,\ \ \  1+1+1'+1,\ \ \  1+1+1+1'.
\end{gather*}
Andrews, Lewis, and Lovejoy denoted the number of partitions with designated summands of size $n$ by the function $PD(n)$.  Hence, using this notation and the example above, we know $PD(4)=10$.

In the same paper, Andrews, Lewis, and Lovejoy \cite{ALL} also considered the restricted partitions with designated summands wherein all parts must be odd, and they denoted the corresponding enumeration function by $PDO(n)$.  Thus, from the example above, we see that $PDO(4)=5$, where we have counted the following five objects:  
$$
3'+1',\ \ \  1'+1+1+1,\ \ \  1+1'+1+1,\ \ \  1+1+1'+1,\ \ \  1+1+1+1'.
$$
Andrews, Lewis, and Lovejoy noted in \cite[eq.~(1.6)]{ALL} that the generating function for $PDO(n)$ is given by 
\begin{equation}
    \label{genfn_main}
    \sum_{n=0}^{\infty} PDO(n)q^n = \frac{E(q^4)E(q^6)^2}{E(q)E(q^3)E(q^{12})},
\end{equation}
where
\begin{align*}
    E(q) = (q;q)_\infty := (1-q)(1-q^{2})(1-q^{3})(1-q^{4})\cdots
\end{align*}
is the usual \emph{$q$-Pochhammer symbol}.  

Beginning with \cite{ALL}, a wide variety of Ramanujan-like congruences have been proven for $PD(n)$ and $PDO(n)$ under different moduli.  See \cite{ACX2018, BK, BO, CJJS, dSS3, dSSk, HBN, NG, Xia} for such work. Recently, Sellers \cite{Sel23} proved a number of arithmetic properties satisfied by $PDO(n)$ modulo powers of $2$ in response to a conjecture of Herden \textit{et.~al.~}\cite{Her}. As part of that work,  Sellers proved the following internal congruences on the way to proving infinite families of divisibility properties satisfied by $PDO(n)$:  
\textit{For all $n\geq 0$}, 
\begin{align}
PDO(4n) &\equiv PDO(n) \pmod{4}, \\
PDO(16n) &\equiv PDO(4n) \pmod{8}. \label{cong_mod_8}  
\end{align}

While searching for such internal congruences computationally, it became clear that the above internal congruences are part of a much larger family.  The ultimate goal of this paper is to prove the existence of the following infinite family:    

\begin{theorem}
\label{mainthm}
For all $k\geq 0$ and all $n\geq 0$, 
\begin{align}
    PDO(2^{2k+3}n) \equiv PDO(2^{2k+1}n) \pmod{2^{2k+3}}.
\end{align}
\end{theorem}

It should be noted that replacing $n$ by $2n$ in the above immediately yields the following corollary:  

\begin{corollary}
 For all $k\geq 0$ and all $n\geq 0$, 
 \begin{align}
     PDO(2^{2k+4}n) \equiv PDO(2^{2k+2}n) \pmod{2^{2k+3}}.
 \end{align}
\end{corollary}

Notice that the $k=0$ case of this corollary states that, for all $n\geq 0$, 
$$
PDO(16n) \equiv PDO(4n) \pmod{8}
$$
which is \eqref{cong_mod_8} mentioned above. It is also necessary to point out that the above families of congruences are not optimal in several isolated cases, especially when $k$ is small. For instance, we will show in \eqref{eq:2^3-2^5} that
$$
PDO(32n) \equiv PDO(8n) \pmod{64},
$$
which also yields
$$
PDO(64n) \equiv PDO(16n) \pmod{64}.
$$

In order to prove Theorem \ref{mainthm}, we introduce the following auxiliary functions:
\begin{align*}
	\delta=\delta(q)&:=\frac{E(q^4) E(q^6)^2}{E(q) E(q^3) E(q^{12})}, \\ 
	\gamma=\gamma(q)&:= \frac{E(q)^5 E(q^2)^5 E(q^6)^5}{E(q^3)^{15}}, \\ 
	\xi=\xi(q)&:=\frac{E(q^2)^5 E(q^6)}{E(q) E(q^3)^5}, 
\end{align*}
and
\begin{align*} 
	\kappa=\kappa(q):=\frac{\gamma(q^2)^2}{\gamma(q)}.
\end{align*}
Note from \eqref{genfn_main} that
\begin{align*}
	\delta(q) = \sum_{n=0}^\infty PDO(n)q^n.
\end{align*}
We further define for $k\ge 2$,
\begin{align}\label{eq:def-L}
	\Lambda_k = \Lambda_k(q):=\gamma(q)^{2^{k-2}}\sum_{n=0}^\infty PDO(2^{k} n)q^n.
\end{align}
Finally, let $U$ be the \emph{unitizing operator of degree two}, given by
\begin{align*}
	U\left(\sum_n a_n q^n\right):= \sum_n a_{2n} q^n.
\end{align*}
These will allow us to represent each $2$-dissection slice of the generating function of $PDO(n)$, accompanied by a certain multiplier:
\begin{align*}
    \lambda_k \sum_{n=0}^\infty PDO(2^k n)q^n,
\end{align*}
as a polynomial in the Hauptmodul $\xi$ on the classical modular curve $X_0(6)$ of genus $0$. We then complete our analysis on these functions in order to prove Theorem \ref{mainthm}.  

This paper is organized as follows: Sect.~\ref{sec:mod-eqn} is devoted to proving several required modular equations. In particular, one important component of this section, as discussed in Sect.~\ref{sec:zeta}, concerns the representation of the degree two unitizations $\zeta_{i,j}=U\big(\kappa^i \xi^j\big)$ in $\mathbb{Z}[\xi]$ for arbitrary exponents $i$ and $j$, where $\kappa$ is one of the auxiliary functions we introduced and $\xi$ is the aforementioned Hauptmodul. The next two sections will then be devoted to the $2$-adic behavior for these $\zeta_{i,j}$ series. Recall that in Theorem \ref{mainthm}, we are indeed considering internal congruences for $PDO(n)$. Hence, we introduce another family of auxiliary functions in Sect.~\ref{sec:new-aux-fnc}:
\begin{align*}
    \lambda'_k \sum_{n=0}^\infty \big(PDO(2^{k+2} n)-PDO(2^k n)\big)q^n,
\end{align*}
so as to capture this internal nature. Here the new multipliers $\lambda'_k$ are closely related to the original $\lambda_k$. We will then move on to the divisibility properties for the above new family of series, and in particular, we offer our proof of Theorem \ref{mainthm} in Sect.~\ref{sec:2-adic-Phi}. It is notable that the construction of each $2$-dissection slice of our generating function bears an entirely different nature to those studied in the past literature, in the sense that the multipliers $\lambda_k$ are pairwise distinct. This fact makes our $2$-adic analysis far more complicated but in the meantime very unique. We present a discussion of the difference between our machinery and past work in Sect.~\ref{sec:conclusion}.

\section{Modular equations}\label{sec:mod-eqn}

In this section, we collect a number of necessary modular relations that will be utilized in the sequel.

\subsection{Modular relations for $\gamma$ and $\xi$}

Let us begin with the functions $\gamma$ and $\xi$. It is clear that both $\gamma^6$ and $\xi$ are modular functions on the classical modular curve $X_0(6)$.

\begin{theorem}\label{th:gamma6-xi-identity}
	We have
	\begin{align}\label{eq:gamma6-xi-identity}
		\gamma^6 &= 59049 \xi^{10} - 262440 \xi^{11} + 466560 \xi^{12} - 414720 \xi^{13}\notag\\
		&\quad + 184320 \xi^{14} - 32768 \xi^{15} .
	\end{align}
\end{theorem}

\begin{proof}
	We only need to analyze the order of $\gamma^6$ and $\xi$ at each of the cusps (cf.~\cite[p.~18, Theorem 1.65]{Ono2004}):
	\begin{align*}
		\ord_0 \gamma^6 &= 5, & \ord_0 \xi &= 0,\\
		\ord_{1/2} \gamma^6 &= 10, & \ord_{1/2} \xi &= 1,\\
		\ord_{1/3} \gamma^6 &= -15, & \ord_{1/3} \xi &= -1,\\
		\ord_\infty \gamma^6 &= 0, & \ord_\infty \xi &= 0.
	\end{align*}
We see that both functions have a pole at the cusp $[\tfrac{1}{3}]_6$, with the pole for $\xi$ being simple. Since $X_0(6)$ has genus $0$, we can write $\gamma^6$ as a polynomial in $\xi$ of degree $15$, as given above. A more streamlined computer-aided analysis could be achieved automatically by Garvan's \textsf{Maple} package \texttt{ETA} \cite{Gar1999}.
\end{proof}

\subsection{Modular relations for $\kappa$ and $\xi$}\label{sec:zeta}

Our objective in this subsection is as follows:

\begin{theorem}\label{th:kx-xi-poly}
	For any $i,j\ge 0$,
	\begin{align}\label{eq:kx-xi-poly}
		U\big(\kappa^i\xi^j\big)\in \mathbb{Z}[\xi].
	\end{align}
\end{theorem}

We briefly postpone the proof of Theorem \ref{th:kx-xi-poly} in order to complete some necessary analysis.  

\subsubsection{Initial cases}

Let us focus on the following five initial relations:

\begin{theorem}\label{th:kx-xi-initial}
	We have
	\begin{align}
		U\big(\kappa\big) &= 5 \xi^3 - 20 \xi^4 + 16 \xi^5,\label{eq:k1x0}\\
		U\big(\xi\big) &= 5 \xi - 4 \xi^2,\label{eq:k0x1}\\
		U\big(\kappa^2\big) &= -\xi^5 + 50 \xi^6 - 400 \xi^7 + 1120 \xi^8 - 1280 \xi^9 + 512 \xi^{10},\label{eq:k2x0}\\
		U\big(\kappa\xi\big) &= 3 \xi^3 - 18 \xi^4 + 16 \xi^5,\label{eq:k1x1}\\
		U\big(\xi^2\big) &= -9 \xi + 58 \xi^2 - 80 \xi^3 + 32 \xi^4.\label{eq:k0x2}
	\end{align}
\end{theorem}

\begin{proof}
	These results can be shown by standard techniques from the theory of modular cusp analysis. Here we shall turn to a computer-aided proof by applying Smoot's \textsf{Mathematica} implementation \texttt{RaduRK} \cite{Smo2021} of the Radu--Kolberg algorithm \cite{Rad2015}. For example, with the \texttt{RaduRK} package, we can express $U\big(\xi\big)$ as a multiplier times a polynomial in a Hauptmodul $t=t(q)$ on $X_0(12)$ that has a pole only at the cusp $[\infty]_{12}$:
	\begin{align*}
		U\big(\xi\big)\cdot \frac{E(q^3)^{12} E(q^4)^4}{q^4 E(q)^4 E(q^{12})^{12}} = 15 t - t^2 + t^3 + t^4,
	\end{align*}
	where
	\begin{align*}
		t = \frac{E(q^4)^4 E(q^6)^2}{q E(q^2)^2 E(q^{12})^4}.
	\end{align*}
	Now to show the above expression equals $5 \xi - 4 \xi^2$ as claimed in \eqref{eq:k0x1}, we only need to examine that the alleged linear combination of eta-products is identical to $0$, i.e.,
	\begin{align*}
		\big(5 \xi - 4 \xi^2\big)-\big(15 t - t^2 + t^3 + t^4\big) \cdot \frac{q^4 E(q)^4 E(q^{12})^{12}}{E(q^3)^{12} E(q^4)^4} = 0.
	\end{align*}
	This task can be performed readily by Garvan's \textsf{Maple} package \texttt{ETA} \cite{Gar1999}, as pointed out in the proof of Theorem \ref{th:gamma6-xi-identity}. The remaining identities can be proven in the same way.  
\end{proof}

\subsubsection{Recurrences}\label{sec:kappa-xi-rec}

We start by considering two generic formal power series $\alpha=\alpha(q)$ and $\beta=\beta(q)$ such that each of $U\big(\alpha\big)$ and $U\big(\alpha^2\big)$ is representable as a \emph{polynomial} with \emph{rational} (usually \emph{integer}) coefficients in a generic formal power series $\rho=\rho(q)$. Now we define
\begin{align*}
	\sigma_{\alpha,1} &:= \alpha(q)+\alpha(-q)\\
	&\;= 2U\big(\alpha\big),\\[6pt]
	\sigma_{\alpha,2} &:= \alpha(q)\alpha(-q)\\
	&\;= \tfrac{1}{2}\big[\big(\alpha(q)+\alpha(-q)\big)^2-\big(\alpha(q)^2+\alpha(-q)^2\big)\big]\\
	&\; = 2U\big(\alpha\big)^2-U\big(\alpha^2\big).
\end{align*}
Clearly, both $\sigma_{\alpha,1}$ and $\sigma_{\alpha,2}$ are in $\mathbb{Q}[\rho]$. Meanwhile, writing $\alpha_0=\alpha(q)$ and $\alpha_1=\alpha(-q)$, then $X=\alpha_0$ and $\alpha_1$ are the two roots of
\begin{align*}
	\big(X-\alpha(q)\big)\big(X-\alpha(-q)\big) = X^2 - \sigma_{\alpha,1} X+ \sigma_{\alpha,2}.
\end{align*}
That is, for $t=0,1$,
\begin{align*}
	\alpha_t^2 - \sigma_{\alpha,1} \alpha_t + \sigma_{\alpha,2} = 0.
\end{align*}
Let $\beta_0=\beta(q)$ and $\beta_1=\beta(-q)$. Now we see that, for $k\ge 2$,
\begin{align*}
	2U\big(\alpha^k \beta\big) &= \alpha_0^k\beta_0 + \alpha_1^k\beta_1\\
	&= \big(\sigma_{\alpha,1}\alpha_0^{k-1}-\sigma_{\alpha,2}\alpha_0^{k-2}\big)\beta_0 + \big(\sigma_{\alpha,1}\alpha_1^{k-1}-\sigma_{\alpha,2}\alpha_1^{k-2}\big)\beta_1\\
	&= 2\sigma_{\alpha,1} U\big(\alpha^{k-1}\beta\big) - 2\sigma_{\alpha,2} U\big(\alpha^{k-2}\beta\big).
\end{align*}
Namely,
\begin{align*}
	U\big(\alpha^k \beta\big) = \sigma_{\alpha,1} U\big(\alpha^{k-1}\beta\big) - \sigma_{\alpha,2} U\big(\alpha^{k-2}\beta\big).
\end{align*}
Hence, if $U\big(\beta\big)$ and $U\big(\alpha\beta\big)$ are also in $\mathbb{Q}[\rho]$, we obtain recursively that $U\big(\alpha^k\beta\big)\in \mathbb{Q}[\rho]$ for all $k\ge 0$.

Now moving back to our scenario, we define for $i,j\ge 0$,
\begin{align}\label{eq:def-zeta}
	\zeta_{i,j} := U\big(\kappa^i\xi^j\big).
\end{align}

\begin{theorem}
	For any $i\ge 2$ and $j\ge 0$,
	\begin{align}\label{eq:zeta-rec-i}
		\zeta_{i,j} = \big(10 \xi^3 - 40 \xi^4 + 32 \xi^5\big)\cdot \zeta_{i-1,j} - \big(\xi^5\big)\cdot \zeta_{i-2,j}.
	\end{align}
	Also, for any $i\ge 0$ and $j\ge 2$,
	\begin{align}\label{eq:zeta-rec-j}
		\zeta_{i,j} = \big(10 \xi - 8 \xi^2\big)\cdot \zeta_{i,j-1} - \big(9 \xi - 8 \xi^2\big)\cdot \zeta_{i,j-2}.
	\end{align}
\end{theorem}

\begin{remark}
	We may combine the two recurrences \eqref{eq:zeta-rec-i} and \eqref{eq:zeta-rec-j} and derive that for $i,j\ge 2$,
	\begin{align}
		\zeta_{i,j} &= \big(10 \xi - 8 \xi^2\big)\big(10 \xi^3 - 40 \xi^4 + 32 \xi^5\big) \cdot \zeta_{i-1,j-1}\notag\\
		&\quad - \big(9 \xi - 8 \xi^2\big)\big(10 \xi^3 - 40 \xi^4 + 32 \xi^5\big) \cdot \zeta_{i-1,j-2}\notag\\
		&\quad - \big(10 \xi - 8 \xi^2\big)\big(\xi^5\big)\cdot \zeta_{i-2,j-1}\notag\\
		&\quad + \big(9 \xi - 8 \xi^2\big)\big(\xi^5\big)\cdot \zeta_{i-2,j-2}.
	\end{align}
\end{remark}

\begin{proof}
	Let us first choose $(\alpha,\beta,\xi)\mapsto (\kappa,\xi^j,\xi)$ with $j\ge 0$. It is clear from \eqref{eq:k1x0} and \eqref{eq:k2x0} that
	\begin{align*}
		\sigma_{\kappa,1}&= 10 \xi^3 - 40 \xi^4 + 32 \xi^5,\\
		\sigma_{\kappa,2}&= \xi^5.
	\end{align*}
	Hence, \eqref{eq:zeta-rec-i} follows. In the same fashion, we choose $(\alpha,\beta,\xi)\mapsto (\xi,\kappa^i,\xi)$ with $i\ge 0$. By \eqref{eq:k1x0} and \eqref{eq:k2x0}, we compute that
	\begin{align*}
		\sigma_{\xi,1}&= 10 \xi - 8 \xi^2,\\
		\sigma_{\xi,2}&= 9 \xi - 8 \xi^2,
	\end{align*}
	and therefore, \eqref{eq:zeta-rec-j} is true.
\end{proof}

Now we are in a position to complete the proof of Theorem \ref{th:kx-xi-poly}.

\begin{proof}[Proof of Theorem \ref{th:kx-xi-poly}]
	With Theorem \ref{th:kx-xi-initial} in hand, we currently have two sets of initial relations for the recurrence \eqref{eq:zeta-rec-i}. Namely, $\{\zeta_{0,0},\zeta_{1,0}\}$ gives us the representation of each $\zeta_{i,0}$ in terms of $\xi$, while $\{\zeta_{0,1},\zeta_{1,1}\}$ gives us the representation of each $\zeta_{i,1}$. With these recipes, \eqref{eq:zeta-rec-j} further produces the representation of each $\zeta_{i,j}$ in $\mathbb{Z}[\xi]$, thereby confirming \eqref{eq:kx-xi-poly}.
\end{proof}

\subsection{Modular relations for $\Lambda$ and $\xi$}

We begin with a relation connecting $\gamma(q^2)\delta(q)^2$ and $\xi(q)$.

\begin{theorem}\label{th:delta-xi}
	We have
	\begin{align}\label{eq:delta-xi}
		U\big(\gamma(q^2)\delta(q)^2\big) = 3 \xi(q)^2 - 2\xi(q)^3 .
	\end{align}
\end{theorem}

\begin{proof}
	This relation can be shown in a way similar to that for Theorem \ref{th:kx-xi-initial}.
\end{proof}

Now we move to relations for $\Lambda_k$ and $\xi$ for each $k\ge 2$.

\begin{theorem}\label{th:L-xi}
	For any $k\ge 2$,
	\begin{align}\label{eq:Lk-xi-poly}
		\Lambda_k \in \mathbb{Z}[\xi].
	\end{align}
	More precisely, if we write
	\begin{align}\label{eq:Lk-xi}
		\Lambda_k:=\sum_{m} c_k(m)\xi^m,
	\end{align}
	then
	\begin{align}\label{eq:L0-xi}
		\Lambda_2 = 3\xi^2-2\xi^3,
	\end{align}
	and for $k\ge 3$, we recursively have
	\begin{align}\label{eq:L-xi-rec}
		\Lambda_k = \sum_{\ell} c_{k-1}(\ell) \cdot \zeta_{2^{k-3},\ell},
	\end{align}
	where $\zeta_{2^{k-3},\ell}$ is given by \eqref{eq:def-zeta}.
\end{theorem}

\begin{proof}
	We begin with the proof of \eqref{eq:L0-xi}. It was already shown in \cite[Theorem 21]{ALL} that
        \begin{align*}
		\sum_{n=0}^\infty PDO(2n)q^n = \delta(q)^2.
	\end{align*}
	Thus,
	\begin{align*}
		\Lambda_2&=\gamma(q)\sum_{n=0}^\infty PDO(4n)q^n = U\left(\gamma(q^2)\sum_{n=0}^\infty PDO(2n)q^n\right) = U\big(\gamma(q^2)\delta(q)^2\big).
	\end{align*}
	Invoking \eqref{eq:delta-xi} gives the claimed expression for $\Lambda_2$.

    For \eqref{eq:L-xi-rec}, we make use of the fact that, for $k\ge 3$,
	\begin{align*}
		\Lambda_k &= \gamma(q)^{2^{k-2}}\sum_{n=0}^\infty PDO(2^{k} n)q^n\\
		&= U\left(\gamma(q^2)^{2^{k-2}}\sum_{n=0}^\infty PDO(2^{k-1} n)q^n\right)\\
		&= U\left(\left(\frac{\gamma(q^2)^{2}}{\gamma(q)}\right)^{2^{k-3}} \gamma(q)^{2^{k-3}}\sum_{n=0}^\infty PDO(2^{k-1} n)q^n\right)\\
		&= U\left(\left(\frac{\gamma(q^2)^{2}}{\gamma(q)}\right)^{2^{k-3}} \Lambda_{k-1}\right).
	\end{align*}
	Noting that $\kappa = \gamma(q^2)^{2}/\gamma(q)$ and writing $\Lambda_{k-1}$ in the above as a polynomial in $\xi$ by virtue of \eqref{eq:Lk-xi}, we finally obtain that
	\begin{align*}
		\Lambda_k = \sum_{\ell} c_{k-1}(\ell) U\left(\kappa^{2^{k-3}}\xi^\ell\right).
	\end{align*}
	Recalling \eqref{eq:def-zeta}, the required result follows.
\end{proof}

\begin{example}
    We give a few examples to illustrate what $\Lambda_k$ looks like when $k\ge 3$:
    \begin{enumerate}[label=\textbf{(\arabic*).~},leftmargin=*,labelsep=0cm,align=left,itemsep=6pt]
        \item In light of the recurrences in Sect.~\ref{sec:kappa-xi-rec}, it is plain that
        \begin{align*}
            \zeta_{1,2} &= -15 \xi^4 + 16 \xi^5,\\
            \zeta_{1,3} &= -27 \xi^4 + 36 \xi^5 - 8 \xi^6.
        \end{align*}
        Hence, from the relation that $\Lambda_2 = 3\xi^2-2\xi^3$, we have
        \begin{align*}
            \Lambda_3 &= 3\zeta_{1,2} - 2\zeta_{1,3}\\
            &= 3\big({-15} \xi^4 + 16 \xi^5\big) -2 \big({-27} \xi^4 + 36 \xi^5 - 8 \xi^6\big),
        \end{align*}
        thereby giving us
        \begin{align}\label{eq:L1-xi}
            \Lambda_3 &= 9 \xi^4 - 24 \xi^5 + 16 \xi^6.
        \end{align}
	
        \item In the same vein,
        \begin{align*}
            \zeta_{2,4} &= -81 \xi^7 + 594 \xi^8 - 1024 \xi^9 + 512 \xi^{10},\\
            \zeta_{2,5} &= 405 \xi^8 - 900 \xi^9 + 496 \xi^{10},\\
            \zeta_{2,6} &= 729 \xi^8 - 1944 \xi^9 + 1728 \xi^{10} - 640 \xi^{11} + 128 \xi^{12}.
    	\end{align*}
    	Therefore,
    	\begin{align} \label{eq:L2-xi}
    		\Lambda_4 &= -729 \xi^7 + 7290 \xi^8 - 18720 \xi^9 + 20352 \xi^{10}\notag\\
            &\quad - 10240 \xi^{11} + 2048 \xi^{12}.
        \end{align}
	
        \item With a lengthier computation, we have
        \begin{align}\label{eq:L3-xi}
            \Lambda_5 &= 34543665 \xi^{14}-400588416 \xi^{15}+2073171024 \xi^{16}-6214952448 \xi^{17}\notag\\
            &\quad +11906611200 \xi^{18}-15261990912 \xi^{19}+13313703936 \xi^{20}\notag\\
            &\quad -7841251328 \xi^{21}+2994733056 \xi^{22}-671088640 \xi^{23}\notag\\
            &\quad +67108864 \xi^{24}.
    	\end{align}
    \end{enumerate}
\end{example}

\section{Minimal $\xi$-power in $\zeta$}

For each $i,j\ge 0$, let the coefficients $Z_{i,j}(m)$ with $m\ge 0$ be such that
\begin{align*}
	\zeta_{i,j}:= \sum_{m} Z_{i,j}(m)\xi^m.
\end{align*}
It is notable that $Z_{i,j}(m)$ eventually vanishes as $\zeta_{i,j}\in \mathbb{Z}[\xi]$, so the above summation is indeed finite. 

From the evaluations in Sect.~\ref{sec:zeta}, we have also seen that for each $\zeta_{i,j}$ as a polynomial in $\xi$, the terms $\xi^m$ with a lower degree usually vanish. In the next theorem, we characterize the minimal $\xi$-power in the polynomial expression of $\zeta_{i,j}$.

\begin{theorem}\label{th:Z-nonvanishing}
	For any $i,j\ge 0$, define
	\begin{align*}
		d_{i,j}:=\begin{cases}
			5I+J, & \text{if $(i,j)=(2I,2J)$},\\
			5I+J+1, & \text{if $(i,j)=(2I,2J+1)$},\\
			5I+J+3, & \text{if $(i,j)=(2I+1,2J)$},\\
			5I+J+3, & \text{if $(i,j)=(2I+1,2J+1)$}.
		\end{cases}
	\end{align*}
	Then for any $m$ with $0\le m< d_{i,j}$, we have
	\begin{align*}
		Z_{i,j}(m) = 0.
	\end{align*}
	Furthermore, the coefficient $Z_{i,j}(d_{i,j})$ is an odd integer.
\end{theorem}

\begin{proof}
	Clearly, $\zeta_{0,0} = 1$, from which we find that it starts with the power $\xi^{0}$ with $Z_{0,0}(d_{0,0}) = Z_{0,0}(0) = 1$ being odd. We further know from \eqref{eq:k1x0} that $\zeta_{1,0}$ starts with the power $5\xi^{3}$, while $Z_{1,0}(d_{1,0}) = Z_{1,0}(3) = 5$ is also odd. Inductively, it follows from the recurrence \eqref{eq:zeta-rec-i} that $\zeta_{2I+2,0}$ starts with the power $Z_{2I+2,0}(5I+5)\xi^{5I+5}$ with
	\begin{align*}
		Z_{2I+2,0}(d_{2I+2,0}) &= Z_{2I+2,0}(5I+5)\\
		& = -Z_{2I,0}(5I),
	\end{align*}
	which is an odd integer. Similarly, we deduce by the same recurrence that $\zeta_{2I+3,0}$ starts with the power $Z_{2I+3,0}(5I+8)\xi^{5I+8}$ with
	\begin{align*}
		Z_{2I+3,0}(d_{2I+3,0}) &= Z_{2I+3,0}(5I+8)\\
		& = 10Z_{2I+2,0}(5I+5)-Z_{2I+1,0}(5I+3),
	\end{align*}
	again being odd.
	
	In the same fashion, we find from \eqref{eq:k0x1} that $\zeta_{0,1}$ starts with the power $5\xi^{1}$ with an odd coefficient $5$, and from \eqref{eq:k1x1} that $\zeta_{1,1}$ starts with the power $3\xi^{3}$ also with its coefficient $3$ being odd. By induction under the rule of the recurrence \eqref{eq:zeta-rec-i}, it follows that $\zeta_{2I+2,1}$ starts with the power $Z_{2I+2,1}(5I+6)\xi^{5I+6}$, where the coefficient is
	\begin{align*}
		Z_{2I+2,1}(5I+6) = 10 Z_{2I+1,1}(5I+3) - Z_{2I,1}(5I+1),
	\end{align*}
	and it is an odd integer. Also, $\zeta_{2I+3,1}$ starts with the power $Z_{2I+3,1}(5I+8)\xi^{5I+8}$, where the coefficient is the odd integer
	\begin{align*}
		Z_{2I+3,1}(5I+8) = - Z_{2I+1,1}(5I+3).
	\end{align*}
	
	Now we have shown the desired result for each $\zeta_{i,0}$ and $\zeta_{i,1}$, and we will then apply induction on $j$. In light of the recurrence \eqref{eq:zeta-rec-j}, we find that $\zeta_{2I,2J+2}$ starts with the power $Z_{2I,2J+2}(5I+J+1)\xi^{5I+J+1}$, where the coefficient is an odd integer given by
	\begin{align*}
		Z_{2I,2J+2}(5I+J+1) = -9Z_{2I,2J}(5I+J).
	\end{align*}
	Likewise, $\zeta_{2I,2J+3}$ starts with the power $Z_{2I,2J+3}(5I+J+2)\xi^{5I+J+2}$, where the coefficient is an odd integer given by
	\begin{align*}
		Z_{2I,2J+3}(5I+J+2) = 10Z_{2I,2J+2}(5I+J+1) - 9Z_{2I,2J+1}(5I+J+1).
	\end{align*}
	Moreover, $\zeta_{2I+1,2J+2}$ starts with the power $Z_{2I+1,2J+2}(5I+J+4)$, where the coefficient is
	\begin{align*}
		Z_{2I+1,2J+2}(5I+J+4) = 10 Z_{2I+1,2J+1}(5I+J+3) - 9 Z_{2I+1,2J}(5I+J+3),
	\end{align*}
	which is odd. Meanwhile, $\zeta_{2I+1,2J+3}$ starts with the power $Z_{2I+1,2J+3}(5I+J+4)$, where the coefficient is
	\begin{align*}
		Z_{2I+1,2J+3}(5I+J+4) = - 9 Z_{2I+1,2J+1}(5I+J+3),
	\end{align*}
	once again being odd.
\end{proof}

\section{$2$-Adic analysis for $\zeta$}

With all of the above preparations in hand, we are now in a position to begin thinking about the divisibility of various objects by powers of $2$. Throughout the remainder of this work, we denote by $\nu(n)$ the \emph{$2$-adic evaluation} of $n$, that is, $\nu(n)$ is the \emph{largest} nonnegative integer $\alpha$ such that $2^\alpha\mid n$. We also adopt the convention that $\nu(0) = \infty$.

The following trivial result on $2$-adic evaluations will be frequently used:

\begin{lemma}
	Let $a$ and $b$ be two integers. Then
	\begin{align}\label{eq:2-adic-add}
		\nu(a+b)\begin{cases}
			\,\,=\,\, \min\{\nu(a),\nu(b)\}, & \text{if $\nu(a)\ne \nu(b)$},\\[6pt]
			\,\,\ge\,\, 2\nu(a)=2\nu(b), & \text{if $\nu(a)= \nu(b)$}.
		\end{cases}
	\end{align}
\end{lemma}

Let us start by analyzing the $2$-adic behavior of the coefficients $Z_{i,0}(m)$ for each $i\ge 0$ whenever $m\ge d_{i,0}$.

\begin{theorem}\label{th:2-adic-Z-i-0}
	For any $i\ge 0$, we have
	\begin{align*}
		\nu\big(Z_{i,0}(d_{i,0})\big) = 0,
	\end{align*}
	and
	\begin{align*} 
		\nu\big(Z_{i,0}(d_{i,0}+1)\big) \begin{cases}
			= 1, & \text{if $i\equiv 2 \bmod{4}$},\\
			\ge 2, & \text{if $i\not\equiv 2 \bmod{4}$}.
		\end{cases}
	\end{align*}
	Furthermore, for $M\ge 2$,
	\begin{align*} 
		\nu\big(Z_{i,0}(d_{i,0}+M)\big) \ge M+1.
	\end{align*}
\end{theorem}

\begin{proof}
	We first note that the theorem is true for $i=0$ and $1$ by the fact that $\zeta_{0,0}=1$ and the relation given in \eqref{eq:k1x0}, respectively. Hence, we may apply induction to prove the required result for $i=4I+2$, $4I+3$, $4I+4$ and $4I+5$ with the assumption that it holds for $i=4I+0$ and $4I+1$.
	
	The $i=4I+2$ case can be illustrated by the following table:
	\begin{align*}
		\renewcommand\arraystretch{1.5}
		\scalebox{0.8}{%
		\begin{tabular*}{1.25\textwidth}{p{0pt}cp{0pt}|@{\extracolsep{\fill}}ccccccccccc}
			\hline
			& $m-d_{4I+0,0}$ && $0$ & $1$ & $2$ & $3$ & $4$ & $5$ & $6$ & $7$ & $8$ & $9$ & $M$\\
			\hline
			& $\nu\big(Z_{4I+0,0}(m)\big)$ && $0$ & $\ge 2$ & $\ge 3$ & $\ge 4$ & $\ge 5$ & $\ge 6$ & $\ge 7$ & $\ge 8$ & $\ge 9$ & $\ge 10$ & $\ge (M+1)$\\
			\hline
			& $\nu\big(Z_{4I+1,0}(m)\big)$ && $\infty$ & $\infty$ & $\infty$ & $0$ & $\ge 2$ & $\ge 3$ & $\ge 4$ & $\ge 5$ & $\ge 6$ & $\ge 7$ & $\ge (M-2)$\\
			\hline
			&  && $\infty$ & $\infty$ & $\infty$ & $\infty$ & $\infty$ & $0$ & $\ge 2$ & $\ge 3$ & $\ge 4$ & $\ge 5$ & $\ge (M-4)$\\
			&  && $\infty$ & $\infty$ & $\infty$ & $\infty$ & $\infty$ & $\infty$ & $1$ & $\ge 3$ & $\ge 4$ & $\ge 5$ & $\ge (M-4)$\\
			& $\nu\big(Z_{4I+2,0}(m)\big)$ && $\infty$ & $\infty$ & $\infty$ & $\infty$ & $\infty$ & $0$ & $1$ & $\ge 3$ & $\ge 4$ & $\ge 5$ & $\ge (M-4)$\\
			\hline
		\end{tabular*}
	}
	\end{align*}
	Here the last column holds for any $M\ge 10$. Recalling the recurrence \eqref{eq:zeta-rec-i}, we have
	\begin{align*}
		\zeta_{4I+2,0} = \big(10 \xi^3 - 40 \xi^4 + 32 \xi^5\big)\cdot \zeta_{4I+1,0} - \big(\xi^5\big)\cdot \zeta_{4I+0,0}.
	\end{align*}
	Hence, the third line provides the $2$-adic evaluations of the coefficients $Z_{4I+2,0}$ contributed from $\big(\xi^5\big)\cdot \zeta_{4I+0,0}$, while the fourth line provides the contribution from $\big(10 \xi^3 - 40 \xi^4 + 32 \xi^5\big)\cdot \zeta_{4I+1,0}$. For example, the $2$-adic evaluation of the coefficient of the $\xi^{d_{4I+0,0}+5}$ term in $\big(\xi^5\big)\cdot \zeta_{4I+0,0}$ is clearly $0$, while in $\big(10 \xi^3 - 40 \xi^4 + 32 \xi^5\big)\cdot \zeta_{4I+1,0}$, the power $\xi^{d_{4I+0,0}+5}$ vanishes, thereby having coefficient $0$ and $2$-adic evaluation $\infty$. Consequently,
	\begin{align*}
		\nu\big(Z_{4I+2,0}(d_{4I+0,0}+5)\big) = 0.
	\end{align*}
	By Theorem \ref{th:Z-nonvanishing}, we have $d_{4I+2,0} = d_{4I+0,0}+5$, so as to give us
	\begin{align*}
		\nu\big(Z_{4I+2,0}(d_{4I+2,0})\big) = 0.
	\end{align*}
	Likewise, the $2$-adic evaluation of the coefficient of the $\xi^{d_{4I+0,0}+6}$ term in $\big(\xi^5\big)\cdot \zeta_{4I+0,0}$ is $\ge 2$, while in $\big(10 \xi^3 - 40 \xi^4 + 32 \xi^5\big)\cdot \zeta_{4I+1,0}$, the corresponding $2$-adic evaluation is $1$. So we have
	\begin{align*}
		\nu\big(Z_{4I+2,0}(d_{4I+0,0}+6)\big) = \min\{\ge 2,1\} = 1,
	\end{align*}
	thereby yielding that
	\begin{align*}
		\nu\big(Z_{4I+2,0}(d_{4I+2,0}+1)\big) = 1.
	\end{align*}
	This process can be continued to all remaining coefficients $Z_{4I+2,0}$.
	
	For $i=4I+3$, we shall make use of the following table and argue in the same vein:
	\begin{align*}
		\renewcommand\arraystretch{1.5}
		\scalebox{0.8}{%
			\begin{tabular*}{1.25\textwidth}{p{0pt}cp{0pt}|@{\extracolsep{\fill}}ccccccccccc}
				\hline
				& $m-d_{4I+1,0}$ && $0$ & $1$ & $2$ & $3$ & $4$ & $5$ & $6$ & $7$ & $8$ & $9$ & $M$\\
				\hline
				& $\nu\big(Z_{4I+1,0}(m)\big)$ && $0$ & $\ge 2$ & $\ge 3$ & $\ge 4$ & $\ge 5$ & $\ge 6$ & $\ge 7$ & $\ge 8$ & $\ge 9$ & $\ge 10$ & $\ge (M+1)$\\
				\hline
				& $\nu\big(Z_{4I+2,0}(m)\big)$ && $\infty$ & $\infty$ & $0$ & $1$ & $\ge 3$ & $\ge 4$ & $\ge 5$ & $\ge 6$ & $\ge 7$ & $\ge 8$ & $\ge (M-1)$\\
				\hline
				&  && $\infty$ & $\infty$ & $\infty$ & $\infty$ & $\infty$ & $0$ & $\ge 2$ & $\ge 3$ & $\ge 4$ & $\ge 5$ & $\ge (M-4)$\\
				&  && $\infty$ & $\infty$ & $\infty$ & $\infty$ & $\infty$ & $1$ & $2$ & $\ge 4$ & $\ge 5$ & $\ge 6$ & $\ge (M-3)$\\
				& $\nu\big(Z_{4I+3,0}(m)\big)$ && $\infty$ & $\infty$ & $\infty$ & $\infty$ & $\infty$ & $0$ & $\ge 2$ & $\ge 3$ & $\ge 4$ & $\ge 5$ & $\ge (M-4)$\\
				\hline
			\end{tabular*}
		}
	\end{align*}

	For $i=4I+4$, the required table is
	\begin{align*}
		\renewcommand\arraystretch{1.5}
		\scalebox{0.8}{%
			\begin{tabular*}{1.25\textwidth}{p{0pt}cp{0pt}|@{\extracolsep{\fill}}ccccccccccc}
				\hline
				& $m-d_{4I+2,0}$ && $0$ & $1$ & $2$ & $3$ & $4$ & $5$ & $6$ & $7$ & $8$ & $9$ & $M$\\
				\hline
				& $\nu\big(Z_{4I+2,0}(m)\big)$ && $0$ & $1$ & $\ge 3$ & $\ge 4$ & $\ge 5$ & $\ge 6$ & $\ge 7$ & $\ge 8$ & $\ge 9$ & $\ge 10$ & $\ge (M+1)$\\
				\hline
				& $\nu\big(Z_{4I+3,0}(m)\big)$ && $\infty$ & $\infty$ & $\infty$ & $0$ & $\ge 2$ & $\ge 3$ & $\ge 4$ & $\ge 5$ & $\ge 6$ & $\ge 7$ & $\ge (M-2)$\\
				\hline
				&  && $\infty$ & $\infty$ & $\infty$ & $\infty$ & $\infty$ & $0$ & $1$ & $\ge 3$ & $\ge 4$ & $\ge 5$ & $\ge (M-4)$\\
				&  && $\infty$ & $\infty$ & $\infty$ & $\infty$ & $\infty$ & $\infty$ & $1$ & $\ge 3$ & $\ge 4$ & $\ge 5$ & $\ge (M-4)$\\
				& $\nu\big(Z_{4I+4,0}(m)\big)$ && $\infty$ & $\infty$ & $\infty$ & $\infty$ & $\infty$ & $0$ & $\ge 2$ & $\ge 3$ & $\ge 4$ & $\ge 5$ & $\ge (M-4)$\\
				\hline
			\end{tabular*}
		}
	\end{align*}
	It is notable that for the $2$-adic evaluation of
	\begin{align*}
		\nu\big(Z_{4I+4,0}(d_{4I+4,0}+1)\big) = \nu\big(Z_{4I+4,0}(d_{4I+2,0}+6)\big),
	\end{align*}
	we shall use the second case of \eqref{eq:2-adic-add} so as to get
	\begin{align*}
		\nu\big(Z_{4I+4,0}(d_{4I+2,0}+6)\big)\ge 2\cdot 1 = 2,
	\end{align*}
	as given in the table.
	
	Finally, for $i=4I+5$, we require the table:
	\begin{align*}
		\renewcommand\arraystretch{1.5}
		\scalebox{0.8}{%
			\begin{tabular*}{1.25\textwidth}{p{0pt}cp{0pt}|@{\extracolsep{\fill}}ccccccccccc}
				\hline
				& $m-d_{4I+3,0}$ && $0$ & $1$ & $2$ & $3$ & $4$ & $5$ & $6$ & $7$ & $8$ & $9$ & $M$\\
				\hline
				& $\nu\big(Z_{4I+3,0}(m)\big)$ && $0$ & $\ge 2$ & $\ge 3$ & $\ge 4$ & $\ge 5$ & $\ge 6$ & $\ge 7$ & $\ge 8$ & $\ge 9$ & $\ge 10$ & $\ge (M+1)$\\
				\hline
				& $\nu\big(Z_{4I+4,0}(m)\big)$ && $\infty$ & $\infty$ & $0$ & $\ge 2$ & $\ge 3$ & $\ge 4$ & $\ge 5$ & $\ge 6$ & $\ge 7$ & $\ge 8$ & $\ge (M-1)$\\
				\hline
				&  && $\infty$ & $\infty$ & $\infty$ & $\infty$ & $\infty$ & $0$ & $\ge 2$ & $\ge 3$ & $\ge 4$ & $\ge 5$ & $\ge (M-4)$\\
				&  && $\infty$ & $\infty$ & $\infty$ & $\infty$ & $\infty$ & $1$ & $\ge 3$ & $\ge 4$ & $\ge 5$ & $\ge 6$ & $\ge (M-3)$\\
				& $\nu\big(Z_{4I+5,0}(m)\big)$ && $\infty$ & $\infty$ & $\infty$ & $\infty$ & $\infty$ & $0$ & $\ge 2$ & $\ge 3$ & $\ge 4$ & $\ge 5$ & $\ge (M-4)$\\
				\hline
			\end{tabular*}
		}
	\end{align*}
	and then perform a similar analysis to that for the above $i=4I+2$ case.
\end{proof}

In the same fashion, we have parallel results for $Z_{i,1}(m)$.

\begin{theorem}\label{th:2-adic-Z-i-1}
	For any $i\ge 0$, we have
	\begin{align*} 
		\nu\big(Z_{i,1}(d_{i,1})\big) = 0,
	\end{align*}
	and
	\begin{align*} 
		\nu\big(Z_{i,1}(d_{i,1}+1)\big) \begin{cases}
			= 1, & \text{if $i\equiv 1 \bmod{4}$},\\
			\ge 2, & \text{if $i\not\equiv 1 \bmod{4}$}.
		\end{cases}
	\end{align*}
	Furthermore, for $M\ge 2$,
	\begin{align*} 
		\nu\big(Z_{i,1}(d_{i,1}+M)\big) \ge M+1.
	\end{align*}
\end{theorem}

\begin{proof}
	By \eqref{eq:k0x1} and \eqref{eq:k1x1}, the theorem holds true for $i=0$ and $1$. We may then apply a similar inductive argument to that for Theorem \ref{th:2-adic-Z-i-0}.  
\end{proof}

Now we are ready to perform the $2$-adic evaluations for the coefficients $Z_{2^k,j}(m)$ for each $k\ge 0$ and $j\ge 0$ whenever $m\ge d_{2^k,j}$.

\begin{theorem}\label{th:Z-2-adic}
	For any $k\ge 2$ and $j\ge 0$, we have
	\begin{align*} 
		\nu\big(Z_{2^k,j}(d_{2^k,j})\big) = 0,
	\end{align*}
	and
	\begin{align*} 
		\nu\big(Z_{2^k,j}(d_{2^k,j}+1)\big) \begin{cases}
			= 1, & \text{if $j\equiv 2 \bmod{4}$},\\
			\ge 2, & \text{if $j\not\equiv 2 \bmod{4}$}.
		\end{cases}
	\end{align*}
	Furthermore, for $M\ge 2$,
	\begin{align*} 
		\nu\big(Z_{2^k,j}(d_{2^k,j}+M)\big) \ge M+1.
	\end{align*}
\end{theorem}

\begin{proof}
	In view of Theorems \ref{th:2-adic-Z-i-0} and \ref{th:2-adic-Z-i-1}, it is known that the results are true for $Z_{2^k,0}(m)$ and $Z_{2^k,1}(m)$ with any $k\ge 2$. Now we apply induction on $j$ and prove for $j=4J+2$, $4J+3$, $4J+4$ and $4J+5$ under the assumption of validity for $j=4J+0$ and $4J+1$. Here a similar strategy to that for Theorem \ref{th:2-adic-Z-i-0} will be used, with \eqref{eq:zeta-rec-j} being invoked:
	\begin{align*}
		\zeta_{2^k,j} = \big(10 \xi - 8 \xi^2\big)\cdot \zeta_{2^k,j-1} - \big(9 \xi - 8 \xi^2\big)\cdot \zeta_{2^k,j-2}.
	\end{align*}
	
	For $j=4J+2$, we require this table:
	\begin{align*}
		\renewcommand\arraystretch{1.5}
		\scalebox{0.8}{%
			\begin{tabular*}{1.25\textwidth}{p{0pt}cp{0pt}|@{\extracolsep{\fill}}ccccccc}
				\hline
				& $m-d_{2^k,4J+0}$ && $0$ & $1$ & $2$ & $3$ & $4$ & $5$ & $M$\\
				\hline
				& $\nu\big(Z_{2^k,4J+0}(m)\big)$ && $0$ & $\ge 2$ & $\ge 3$ & $\ge 4$ & $\ge 5$ & $\ge 6$ & $\ge (M+1)$\\
				\hline
				& $\nu\big(Z_{2^k,4J+1}(m)\big)$ && $\infty$ & $0$ & $\ge 2$ & $\ge 3$ & $\ge 4$ & $\ge 5$ & $\ge (M+0)$\\
				\hline
				&  && $\infty$ & $0$ & $\ge 2$ & $\ge 3$ & $\ge 4$ & $\ge 5$ & $\ge (M+0)$\\
				&  && $\infty$ & $\infty$ & $1$ & $\ge 3$ & $\ge 4$ & $\ge 5$ & $\ge (M+0)$\\
				& $\nu\big(Z_{2^k,4J+2}(m)\big)$ && $\infty$ & $0$ & $1$ & $\ge 3$ & $\ge 4$ & $\ge 5$ & $\ge (M+0)$\\
				\hline
			\end{tabular*}
		}
	\end{align*}
	For $j=4J+3$, we require this table:
	\begin{align*}
		\renewcommand\arraystretch{1.5}
		\scalebox{0.8}{%
			\begin{tabular*}{1.25\textwidth}{p{0pt}cp{0pt}|@{\extracolsep{\fill}}ccccccc}
				\hline
				& $m-d_{2^k,4J+1}$ && $0$ & $1$ & $2$ & $3$ & $4$ & $5$ & $M$\\
				\hline
				& $\nu\big(Z_{2^k,4J+1}(m)\big)$ && $0$ & $\ge 2$ & $\ge 3$ & $\ge 4$ & $\ge 5$ & $\ge 6$ & $\ge (M+1)$\\
				\hline
				& $\nu\big(Z_{2^k,4J+2}(m)\big)$ && $0$ & $1$ & $\ge 3$ & $\ge 4$ & $\ge 5$ & $\ge 6$ & $\ge (M+1)$\\
				\hline
				&  && $\infty$ & $0$ & $\ge 2$ & $\ge 3$ & $\ge 4$ & $\ge 5$ & $\ge (M+0)$\\
				&  && $\infty$ & $1$ & $2$ & $\ge 4$ & $\ge 5$ & $\ge 6$ & $\ge (M+1)$\\
				& $\nu\big(Z_{2^k,4J+3}(m)\big)$ && $\infty$ & $0$ & $\ge 2$ & $\ge 3$ & $\ge 4$ & $\ge 5$ & $\ge (M+0)$\\
				\hline
			\end{tabular*}
		}
	\end{align*}
	For $j=4J+4$, we require this table:
	\begin{align*}
		\renewcommand\arraystretch{1.5}
		\scalebox{0.8}{%
			\begin{tabular*}{1.25\textwidth}{p{0pt}cp{0pt}|@{\extracolsep{\fill}}ccccccc}
				\hline
				& $m-d_{2^k,4J+2}$ && $0$ & $1$ & $2$ & $3$ & $4$ & $5$ & $M$\\
				\hline
				& $\nu\big(Z_{2^k,4J+2}(m)\big)$ && $0$ & $1$ & $\ge 3$ & $\ge 4$ & $\ge 5$ & $\ge 6$ & $\ge (M+1)$\\
				\hline
				& $\nu\big(Z_{2^k,4J+3}(m)\big)$ && $\infty$ & $0$ & $\ge 2$ & $\ge 3$ & $\ge 4$ & $\ge 5$ & $\ge (M+0)$\\
				\hline
				&  && $\infty$ & $0$ & $1$ & $\ge 3$ & $\ge 4$ & $\ge 5$ & $\ge (M+0)$\\
				&  && $\infty$ & $\infty$ & $1$ & $\ge 3$ & $\ge 4$ & $\ge 5$ & $\ge (M+0)$\\
				& $\nu\big(Z_{2^k,4J+4}(m)\big)$ && $\infty$ & $0$ & $\ge 2$ & $\ge 3$ & $\ge 4$ & $\ge 5$ & $\ge (M+0)$\\
				\hline
			\end{tabular*}
		}
	\end{align*}
	For $j=4J+5$, we require this table:
	\begin{align*}
		\renewcommand\arraystretch{1.5}
		\scalebox{0.8}{%
			\begin{tabular*}{1.25\textwidth}{p{0pt}cp{0pt}|@{\extracolsep{\fill}}ccccccc}
				\hline
				& $m-d_{2^k,4J+3}$ && $0$ & $1$ & $2$ & $3$ & $4$ & $5$ & $M$\\
				\hline
				& $\nu\big(Z_{2^k,4J+3}(m)\big)$ && $0$ & $\ge 2$ & $\ge 3$ & $\ge 4$ & $\ge 5$ & $\ge 6$ & $\ge (M+1)$\\
				\hline
				& $\nu\big(Z_{2^k,4J+4}(m)\big)$ && $0$ & $\ge 2$ & $\ge 3$ & $\ge 4$ & $\ge 5$ & $\ge 6$ & $\ge (M+1)$\\
				\hline
				&  && $\infty$ & $0$ & $\ge 2$ & $\ge 3$ & $\ge 4$ & $\ge 5$ & $\ge (M+0)$\\
				&  && $\infty$ & $1$ & $\ge 3$ & $\ge 4$ & $\ge 5$ & $\ge 6$ & $\ge (M+1)$\\
				& $\nu\big(Z_{2^k,4J+5}(m)\big)$ && $\infty$ & $0$ & $\ge 2$ & $\ge 3$ & $\ge 4$ & $\ge 5$ & $\ge (M+0)$\\
				\hline
			\end{tabular*}
		}
	\end{align*}
    Concrete analyses can be mimicked by consulting the $i=4I+2$ case in the proof of Theorem \ref{th:2-adic-Z-i-0}, and we will omit the details.
\end{proof}

\section{New auxiliary functions and the associated minimal $\xi$-powers}\label{sec:new-aux-fnc}

All of the work above has revolved around the generating function for the function $PDO(n)$. However, we keep in mind that Theorem \ref{mainthm} is really focused on the internal congruences for the $PDO$ function. To capture this nature, let us introduce a new family of auxiliary functions for $k\ge 3$,
\begin{align*}
    \Phi_k=\Phi_k(q):=\gamma(q)^{2^{k}}\left(\sum_{n=0}^\infty PDO(2^{k+2}n)q^n - \sum_{n=0}^\infty PDO(2^{k}n)q^n\right).
\end{align*}
In light of \eqref{eq:def-L}, we have
\begin{align*} 
	\Phi_k &= \Lambda_{k+2} - \gamma^{3\cdot 2^{k-2}} \Lambda_k\notag\\
    &= \Lambda_{k+2} - \left(\gamma^6 \right)^{ 2^{k-3}} \Lambda_k.
\end{align*}

\begin{theorem}\label{th:F-xi}
	For any $k\ge 3$,
	\begin{align} \label{eq:Fk-xi-poly}
		\Phi_k \in \mathbb{Z}[\xi].
	\end{align}
	More precisely, if we write
	\begin{align}\label{eq:Fk-xi}
		\Phi_k:=\sum_{m} F_k(m)\xi^m,
	\end{align}
	then
	\begin{align}\label{eq:F1-xi}
		\Phi_3 &= 34012224 \xi^{14} - 396809280 \xi^{15} + 2061728640 \xi^{16} - 6195823488 \xi^{17}\notag\\
		&\quad + 11887534080 \xi^{18} - 15250636800 \xi^{19} + 13309968384 \xi^{20}\notag\\
		&\quad  - 7840727040 \xi^{21} + 2994733056 \xi^{22} - 671088640 \xi^{23}\notag\\
		&\quad  + 67108864 \xi^{24},
	\end{align}
	and for $k\ge 4$, we recursively have
	\begin{align}\label{eq:F-xi-rec}
		\Phi_k = \sum_{\ell} F_{k-1}(\ell) \cdot \zeta_{2^{k-1},\ell},
	\end{align}
	where $\zeta_{2^{k-1},\ell}$ is given by \eqref{eq:def-zeta}.
\end{theorem}

\begin{proof}
	Since both $\Lambda_k$ and $\Lambda_{k+2}$ are in $\mathbb{Z}[\xi]$ by \eqref{eq:Lk-xi-poly}, while $\gamma^6$ is also in $\mathbb{Z}[\xi]$ as we have shown in \eqref{eq:gamma6-xi-identity}, it follows that $\Phi_k \in \mathbb{Z}[\xi]$. In particular, $\Phi_3 = \Lambda_5 - \gamma^6 \Lambda_3$. Applying \eqref{eq:L1-xi} and \eqref{eq:L3-xi} together with \eqref{eq:gamma6-xi-identity} gives \eqref{eq:F1-xi}. For \eqref{eq:F-xi-rec}, we note that
	\begin{align*}
		\Phi_k&=U\left(\gamma(q^2)^{2^{k}}\left(\sum_{n=0}^\infty PDO(2^{k+1}n)q^n - \sum_{n=0}^\infty PDO(2^{k-1}n)q^n\right)\right)\\
		&= U\left(\left(\frac{\gamma(q^2)^2}{\gamma(q)}\right)^{2^{k-1}} \gamma(q)^{2^{k-1}}\left(\sum_{n=0}^\infty PDO(2^{k+1}n)q^n - \sum_{n=0}^\infty PDO(2^{k-1}n)q^n\right)\right)\\
		&= U\left(\kappa^{2^{k-1}}\Phi_{k-1}\right).
	\end{align*}
	Finally, we recall \eqref{eq:Fk-xi} and invoke \eqref{eq:def-zeta} to obtain the desired recurrence.
\end{proof}

Next, in analogy to Theorem \ref{th:Z-nonvanishing}, we show that the lower powers of $\xi$ in each $\Phi_k$ shall vanish.

\begin{theorem}\label{th:F-nonvanishing}
	For any $k\ge 3$, define
	\begin{align*}
		\tau_{k}:=\begin{cases}
			7\cdot 2^{2K-3}-\tfrac{2}{3}\big(4^{K-2}-1\big), & \text{if $k=2K-1$},\\[6pt]
			7\cdot 2^{2K-2}-\tfrac{1}{3}\big(4^{K-1}-1\big), & \text{if $k=2K$}.
		\end{cases}
	\end{align*}
	Then for any $m$ with $0\le m< \tau_k$, we have
	\begin{align*}
		F_k(m) = 0.
	\end{align*}
\end{theorem}

\begin{remark}\label{rem:tau-mod-4}
	A straightforward computation reveals that 
	\begin{align}
		\tau_{2K-1} &\equiv 2 \pmod{4},\label{eq:tau-odd-mod-4}\\
		\tau_{2K} &\equiv 3 \pmod{4},\label{eq:tau-even-mod-4}
	\end{align}
	for any choice of $K\ge 2$.
\end{remark}

\begin{proof}
	From \eqref{eq:F1-xi}, we see that $\Phi_3$ starts with the power $\xi^{14}$, while
	\begin{align*}
		14 = 7\cdot 2^1 - \tfrac{2}{3}\big(4^0-1\big).
	\end{align*}
	Now we assume that the theorem is true for a certain $k=2K-1$, and we prove inductively for $k=2K$ and $2K+1$.
	
	For $k=2K$, we know from \eqref{eq:F-xi-rec} that
	\begin{align}\label{eq:Phi-2K}
		\Phi_{2K} = F_{2K-1}(\tau_{2K-1}) \cdot \zeta_{2^{2K-1},\tau_{2K-1}} + \cdots.
	\end{align}
	By virtue of Theorem \ref{th:Z-nonvanishing}, $\zeta_{2^{2K-1},\tau_{2K-1}}$ starts with the power $\xi^{d}$ where
	\begin{align*}
		d&=d_{2^{2K-1},\tau_{2K-1}}\\
		&=5\cdot 2^{2K-2}+\tfrac{1}{2}\tau_{2K-1}\\
		&=5\cdot 2^{2K-2}+7\cdot 2^{2K-4}-\tfrac{1}{3}\big(4^{K-2}-1\big)\\
		&=7\cdot 2^{2K-2}-\tfrac{1}{3}\big(4^{K-1}-1\big)\\
		&=\tau_{2K}.
	\end{align*}
	Here we make use of the fact that both $2^{2K-1}$ and $\tau_{2K-1}$ are even. In the meantime, each of the $\zeta \in \mathbb{Z}[\xi]$ component in the remaining summands in \eqref{eq:Phi-2K} starts with a power higher than $\xi^d$. Hence, $\Phi_{2K}$ starts with \emph{at least} $\xi^{\tau_{2K}}$.
	
	For $k=2K+1$, we also deduce from \eqref{eq:F-xi-rec} that
	\begin{align}\label{eq:Phi-2K+1}
		\Phi_{2K+1} = F_{2K}(\tau_{2K}) \cdot \zeta_{2^{2K},\tau_{2K}} + F_{2K}(\tau_{2K}+1) \cdot \zeta_{2^{2K},\tau_{2K}+1} + \cdots.
	\end{align}
	Invoking Theorem \ref{th:Z-nonvanishing} and noting that $\tau_{2K}$ is odd, it follows that $\zeta_{2^{2K},\tau_{2K}}$ starts with the power $\xi^{d}$ where
	\begin{align*}
		d&=d_{2^{2K},\tau_{2K}}\\
		&=5\cdot 2^{2K-1}+\tfrac{1}{2}\big(\tau_{2K}-1\big)+1\\
		&=5\cdot 2^{2K-1}+7\cdot 2^{2K-3}-\tfrac{1}{6}\big(4^{K-1}-1\big)+\tfrac{1}{2}\\
		&=7\cdot 2^{2K-1}-\tfrac{2}{3}\big(4^{K-1}-1\big)\\
		&=\tau_{2K+1}.
	\end{align*}
	For the second summand in \eqref{eq:Phi-2K+1}, we find that $\zeta_{2^{2K},\tau_{2K}+1}$ starts with the power $\xi^{d'}$ where
	\begin{align*}
		d'&=d_{2^{2K},\tau_{2K}+1}\\
		&=5\cdot 2^{2K-1}+\tfrac{1}{2}\big(\tau_{2K}+1\big)\\
		&=\tau_{2K+1},
	\end{align*}
	according to a similar computation. Furthermore, each of the $\zeta \in \mathbb{Z}[\xi]$ components in the remaining summands in \eqref{eq:Phi-2K+1} starts with a power higher than $\xi^d=\xi^{d'}$. Hence, we can claim that $\Phi_{2K+1}$ starts with \emph{at least}  $\xi^{\tau_{2K+1}}$.
\end{proof}

\section{$2$-Adic analysis for $\Phi$}\label{sec:2-adic-Phi}

We are now in a position to prove our main result, Theorem \ref{mainthm}. To do so, we only need to confirm that for each $K\ge 1$,
\begin{align*}
	\nu\big(F_{2K+1}(m)\big)\ge 2K+3
\end{align*}
whenever $m\ge \tau_{2K+1}$. In what follows, we first manually analyze the initial cases where $K\in\{1,2\}$ and then move on to general $K$ by induction.

\subsection{Initial cases}\label{sec:Phi-2-adic-initial}

Recall that $\Phi_3$ was already formulated in \eqref{eq:F1-xi}. Meanwhile, we may obtain an explicit expression in $\mathbb{Z}[\xi]$ (containing $43$ terms!) for $\Phi_5$ by applying the recurrence \eqref{eq:F-xi-rec} twice. Consequently, we get the following $2$-adic evaluations, where $\tau$ stands for $\tau_3=14$ or $\tau_5=54$ accordingly:
\begin{align*}
	\renewcommand\arraystretch{1.5}
	\begin{tabular*}{\textwidth}{p{1pt}cp{1pt}|@{\extracolsep{\fill}}cccc}
		\hline
		& $m-\tau$ && $0$ & $1$ & $2$ & $M$ \\
		\hline
		& $\nu\big(F_3(m)\big)$ && $6$ & $6$ & $7$ & $\ge (M+4)$ \\
		\hline
		& $\nu\big(F_5(m)\big)$ && $7$ & $7$ & $9$ & $\ge (M+8)$ \\
		\hline
	\end{tabular*}
\end{align*}
In the above table, the last column is true for all $M\ge 3$. So we indeed have the following two congruences, with the former being even stronger than the general scenario in Theorem \ref{mainthm}:
\begin{align}
	PDO(2^3n) \equiv PDO(2^5n) \pmod{2^6},\label{eq:2^3-2^5}\\
	PDO(2^5n) \equiv PDO(2^7n) \pmod{2^7}.\label{eq:2^5-2^7}
\end{align}

\subsection{Induction}

Now we perform induction on $K\ge 2$ and establish the following lower bounds for the $2$-adic evaluations.

\begin{theorem}
	For any $K\ge 2$, it is true that
	\begin{align}
		\nu\big(F_{2K+1}(\tau_{2K+1})\big) &\ge 2K+3,\label{eq:F-nu-0}
	\end{align}
	and that for $M\ge 1$,
	\begin{align}\label{eq:F-nu-m}
		\nu\big(F_{2K+1}(\tau_{2K+1}+M)\big) &\ge 2K+M+2.
	\end{align}
\end{theorem}

Note that the $K=2$ case was already covered in Sect.~\ref{sec:Phi-2-adic-initial}. So in the sequel, we first prove the inequality \eqref{eq:F-nu-m} for $K\ge 3$ under the inductive assumption that it is true for $K-1$. For the sake of notational convenience, let us denote
\begin{align*}
	\tau'':=\tau_{2K-1},\qquad \tau':=\tau_{2K}, \qquad \tau:=\tau_{2K+1}.
\end{align*}
Also, we write
\begin{align*}
	\{Z''(j,m):j,m\ge 0\}&:=\{Z_{2^{2K-1},j}(m):j,m\ge 0\},\\
	\{Z'(j,m):j,m\ge 0\}&:=\{Z_{2^{2K},j}(m):j,m\ge 0\}.
\end{align*}

\begin{proof}[Proof of \eqref{eq:F-nu-m}]
	To produce $\Phi_{2K+1}$ from $\Phi_{2K-1}$, we need to employ the recurrence \eqref{eq:F-xi-rec} twice. Writing
	\begin{align*}
		\Phi_{2K-1} = \sum_{M''\ge 0} F_{2K-1}(\tau'' + M'')\xi^{\tau''+M''},
	\end{align*}
	we deduce from \eqref{eq:F-xi-rec} that
	\begin{align*}
		\Phi_{2K} = \sum_{M''\ge 0} F_{2K-1}(\tau'' + M'')\zeta_{2^{2K-1},\tau'' + M''}.
	\end{align*}
	Recall that $\tau''$ is even according to \eqref{eq:tau-odd-mod-4}. It then follows from Theorems \ref{th:Z-nonvanishing} and \ref{th:F-nonvanishing} that the minimal degree of the $\xi$-powers in $\zeta_{2^{2K-1},\tau'' + M''}$ is $\tau'+\lfloor\tfrac{M''+1}{2}\rfloor$, i.e.,
	\begin{align}\label{eq:d-tau-2}
		d_{2^{2K-1},\tau'' + M''} = \tau'+\lfloor\tfrac{M''+1}{2}\rfloor.
	\end{align}
	Hence,
	\begin{align*}
		\Phi_{2K} = \sum_{\substack{M''\ge 0\\M'\ge \lfloor\frac{M''+1}{2}\rfloor}} F_{2K-1}(\tau'' + M'')Z''(\tau'' + M'',\tau'+M')\xi^{\tau'+M'}.
	\end{align*}
	Applying the recurrence \eqref{eq:F-xi-rec} once again gives us that
	\begin{align*}
		\Phi_{2K+1} = \sum_{\substack{M''\ge 0\\M'\ge \lfloor\frac{M''+1}{2}\rfloor}} F_{2K-1}(\tau'' + M'')Z''(\tau'' + M'',\tau'+M')\zeta_{2^{2K},\tau' + M'}.
	\end{align*}
	Noting that $\tau'$ is odd according to \eqref{eq:tau-even-mod-4}, Theorems \ref{th:Z-nonvanishing} and \ref{th:F-nonvanishing} then imply that the minimal degree of the $\xi$-powers in $\zeta_{2^{2K},\tau' + M'}$ is $\tau+\lfloor\tfrac{M'}{2}\rfloor$, i.e.,
	\begin{align}\label{eq:d-tau-1}
		d_{2^{2K},\tau' + M'} = \tau'+\lfloor\tfrac{M'}{2}\rfloor.
	\end{align}
	Hence,
	\begin{align*}
		\Phi_{2K+1} = \sum_{\substack{M''\ge 0\\M'\ge \lfloor\frac{M''+1}{2}\rfloor\\M\ge \lfloor\frac{M'}{2}\rfloor}} C(M,M'',M'') \xi^{\tau+M},
	\end{align*}
	where
	\begin{align*}
		C(M,M'',M''):=F_{2K-1}(\tau'' + M'')Z''(\tau'' + M'',\tau'+M')Z'(\tau' + M',\tau+M).
	\end{align*}
	
	Note that to complete the inductive step, it is sufficient to show the inequality
	\begin{align}\label{eq:nu-C}
		\nu\big(C(M,M'',M'')\big) \ge 2K+M+2
	\end{align}
	for every tuple $(M,M'',M'')$ such that $M''\ge 0$, $M'\ge \lfloor\tfrac{M''+1}{2}\rfloor$ and $M\ge \lfloor\tfrac{M'}{2}\rfloor$.
	
	We first consider the following two generic cases where \textbf{(i).}~$M''\ge 2$ and \textbf{(ii).}~$M''=0$ or $1$ and $M'\ge 3$. By our inductive assumption,
	\begin{align*}
		\nu\big(F_{2K-1}(\tau'' + M'')\big) \ge 2K + M''.
	\end{align*}
	Also, by a weaker form of Theorem \ref{th:Z-2-adic} with recourse to \eqref{eq:d-tau-2} and \eqref{eq:d-tau-1}, respectively, we have
	\begin{align*}
		\nu\big(Z''(\tau'' + M'',\tau'+M')\big) &\ge M'-\lfloor\tfrac{M''+1}{2}\rfloor,\\
		\nu\big(Z'(\tau' + M',\tau+M)\big) &\ge M-\lfloor\tfrac{M'}{2}\rfloor.
	\end{align*}
	It follows that
	\begin{align*}
		\nu\big(C(M,M'',M'')\big) \ge 2K+M+\big(M''-\lfloor\tfrac{M''+1}{2}\rfloor\big) + \big(M'-\lfloor\tfrac{M'}{2}\rfloor\big).
	\end{align*}
	Now \textbf{(i).}~if $M''\ge 2$, then $M'\ge 1$ so that $M''-\lfloor\tfrac{M''+1}{2}\rfloor\ge 1$ and that $M'-\lfloor\tfrac{M'}{2}\rfloor\ge 1$; \textbf{(ii).}~if $M''=0$ or $1$ and $M'\ge 3$, then $M''-\lfloor\tfrac{M''+1}{2}\rfloor= 0$ but $M'-\lfloor\tfrac{M'}{2}\rfloor\ge 2$. Thus, in both cases the inequality \eqref{eq:nu-C} holds.
	
	Now there are five sporadic cases to be investigated, and we work on them separately. Note that for \eqref{eq:F-nu-m}, we have the condition that $M\ge 1$.
	\begin{enumerate}[label={\textbf{(\alph*).}},leftmargin=*,labelsep=0cm,align=left]
		\item $(M',M'')=(0,0)$: In this case, it is known by the inductive assumption,
		\begin{align*}
			\nu\big(F_{2K-1}(\tau'' + 0)\big) \ge 2K + 1,
		\end{align*}
		that by Theorem \ref{th:Z-2-adic},
		\begin{align*}
			\nu\big(Z''(\tau'' + 0,\tau'+0)\big) = 0,
		\end{align*}
		and that also by Theorem \ref{th:Z-2-adic} (noting that $\tau'+0\equiv 3\pmod{4}$),
		\begin{align*}
			\nu\big(Z'(\tau' + 0,\tau+M)\big) \ge M+1.
		\end{align*}
	
		\item $(M',M'')=(1,0)$: In this case, it is known by the inductive assumption,
		\begin{align*}
			\nu\big(F_{2K-1}(\tau'' + 0)\big) \ge 2K + 1,
		\end{align*}
		that by Theorem \ref{th:Z-2-adic} (noting that $\tau''+0\equiv 2 \pmod{4}$),
		\begin{align*}
			\nu\big(Z''(\tau'' + 0,\tau'+1)\big) = 1,
		\end{align*}
		and that also by Theorem \ref{th:Z-2-adic} (noting that $\tau'+1\equiv 0\pmod{4}$),
		\begin{align*}
			\nu\big(Z'(\tau' + 1,\tau+M)\big) \ge M+1.
		\end{align*}
	
		\item $(M',M'')=(2,0)$: In this case, it is known by the inductive assumption,
		\begin{align*}
			\nu\big(F_{2K-1}(\tau'' + 0)\big) \ge 2K + 1,
		\end{align*}
		that by Theorem \ref{th:Z-2-adic},
		\begin{align*}
			\nu\big(Z''(\tau'' + 0,\tau'+2)\big) \ge 3,
		\end{align*}
		and that by a weaker form of Theorem \ref{th:Z-2-adic},
		\begin{align*}
			\nu\big(Z'(\tau' + 2,\tau+M)\big) \ge M-1.
		\end{align*}
	
		\item $(M',M'')=(1,1)$: In this case, it is known by the inductive assumption,
		\begin{align*}
			\nu\big(F_{2K-1}(\tau'' + 1)\big) \ge 2K + 1,
		\end{align*}
		that by Theorem \ref{th:Z-2-adic},
		\begin{align*}
			\nu\big(Z''(\tau'' + 1,\tau'+1)\big) = 0,
		\end{align*}
		and that also by Theorem \ref{th:Z-2-adic} (noting that $\tau'+1\equiv 0\pmod{4}$),
		\begin{align*}
			\nu\big(Z'(\tau' + 1,\tau+M)\big) \ge M+1.
		\end{align*}
	
		\item $(M',M'')=(2,1)$: In this case, it is known by the inductive assumption,
		\begin{align*}
			\nu\big(F_{2K-1}(\tau'' + 1)\big) \ge 2K + 1,
		\end{align*}
		that by Theorem \ref{th:Z-2-adic} (noting that $\tau''+1\equiv 3 \pmod{4}$),
		\begin{align*}
			\nu\big(Z''(\tau'' + 1,\tau'+2)\big) \ge 2,
		\end{align*}
		and that by a weaker form of Theorem \ref{th:Z-2-adic},
		\begin{align*}
			\nu\big(Z'(\tau' + 2,\tau+M)\big) \ge M-1.
		\end{align*}
	\end{enumerate}
	We find that in each of these cases, the inequality \eqref{eq:nu-C} is also valid, thereby closing our proof of \eqref{eq:F-nu-m}.
\end{proof}

Finally, for \eqref{eq:F-nu-0}, we shall make use of completely different logic without consulting the relevant recurrences.

\begin{proof}[Proof of \eqref{eq:F-nu-0}]
	Recall that we have shown that the $2$-adic evaluations satisfy
	\begin{align*}
		\nu\big(F_{2K+1}(\tau_{2K+1}+M)\big)\ge 2^{2K+3}
	\end{align*}
	whenever $M\ge 1$ in the proof of \eqref{eq:F-nu-m}. Modulo $2^{2K+3}$, we have
	\begin{align}\label{eq:Phi-2k+1-cong-2k+5}
		\Phi_{2K+1} \equiv F_{2K+1}(\tau_{2K+1})\cdot \xi^{\tau_{2K+1}} \pmod{2^{2K+3}}.
	\end{align}
	Note that the constant term on the left-hand side is $0$ since
	\begin{align*}
		\Phi_{2K+1}=\gamma(q)^{2^{2K+1}}\left(\sum_{n=0}^\infty PDO(2^{2K+3}n)q^n - \sum_{n=0}^\infty PDO(2^{2K+1}n)q^n\right),
	\end{align*}
	while within the parentheses it is clear that
	\begin{align*}
		PDO(2^{2K+3}\cdot 0)q^0 - PDO(2^{2K+1}\cdot 0)q^0 = 0.
	\end{align*}
	Meanwhile, the constant term on the right-hand side of \eqref{eq:Phi-2k+1-cong-2k+5} is $F_{2K+1}(\tau_{2K+1})$ as the constant term of $\xi(q)$ is $1$. Since the congruence \eqref{eq:Phi-2k+1-cong-2k+5} is valid, we must have that their constant terms are also congruent under the same modulus, so as to give us
	\begin{align*}
		F_{2K+1}(\tau_{2K+1}) \equiv 0 \pmod{2^{2K+3}},
	\end{align*}
	which is equivalent to \eqref{eq:F-nu-0}.
\end{proof}

\section{Conclusion}\label{sec:conclusion}

We close with two sets of comments. First, we remind the reader of one of the original goals for proving such internal congruences.  As was mentioned above, the second author \cite{Sel23} used two corollaries of Theorem \ref{mainthm} to prove the following Ramanujan-like congruences via induction on $\alpha$:
\begin{theorem}
For all $\alpha \geq 0$ and all $n\geq 0,$
\begin{align*}
PDO\big(2^\alpha(4n+3)\big) &\equiv 0 \pmod{4}, \\
PDO\big(2^\alpha(8n+7)\big) &\equiv 0 \pmod{8}.
\end{align*}
\end{theorem}
\noindent It would be gratifying to see other cases of Theorem \ref{mainthm} used to assist in proving divisibility properties satisfied by $PDO(n)$ for higher powers of $2$.

Secondly, in the course of proving congruences modulo arbitrary powers for the coefficients of an eta-product
\begin{align*}
	\mathrm{H}(q) = \sum_{n=0}^\infty h(n)q^n,
\end{align*}
the usual strategy is to find a suitable basis $\{\xi_1,\xi_2,\ldots,\xi_L\}$ of the corresponding modular space such that each dissection slice, accompanied by a certain multiplier (usually an eta-product),
\begin{align*}
	\lambda_m \sum_{n=0}^\infty h(p^m n + t_m) q^n,
\end{align*}
can be represented as a polynomial in $\mathbb{Z}[\xi_1,\xi_2,\ldots,\xi_L]$. For example, when proving the congruences modulo powers of 5 for the partition function \cite{Wat1938} or the congruences modulo powers of $7$ for the distinct partition function \cite{GH1981}, two specific multipliers typically take turns showing up, i.e., $\lambda_{2M-1}=\lambda$ and $\lambda_{2M}=\lambda'$ for two certain series $\lambda$ and $\lambda'$. However, in our study here, the multipliers $\gamma$, $\gamma^2$, $\gamma^4$, $\gamma^8$, $\ldots$~never overlap. Meanwhile, an important outcome of cycling the multipliers in the previous studies is that it is typically sufficient to represent each degree $p$ unitization $U_p\big(\kappa^i \xi^j\big)$ as a polynomial in $\xi$ for a certain series $\kappa$, with the exponent $i$ restricted to $\{0,1\}$; here we use the case where the basis of the modular space is given by $\{\xi\}$ as an illustration. However, as shown in Sect.~\ref{sec:zeta}, when there are endless possibilities for the multipliers, we have to extend the consideration of $i$ to infinity, thereby substantially increasing the amount of required $p$-adic analysis. These striking facts distinguish our work from the past literature.

\end{document}